\documentclass{amsart}
\usepackage{amsfonts}
\usepackage{amsmath,amssymb}	
\usepackage{amsthm}
\usepackage{amscd}
\usepackage{graphics}
\usepackage{graphicx}
{
\theoremstyle{definition}
\newtheorem{Def}{Definition}

}

\newtheorem{Thm}{Theorem}
\newtheorem{Lem}{Lemma}
\begin{document}
\title[Products of cobordism classes of Morse functions]{On defining products of cobordism classes of Morse functions}

\author{Naoki Kitazawa}
\keywords{Singularities of differentiable maps; generic maps. Differential topology. Reeb spaces}
\subjclass[2010]{Primary~57R45. Secondary~57N15.}
\address{Institute of Mathematics for Industry, Kyushu University, 744 Motooka, Nishi-ku Fukuoka 819-0395, Japan}
\email{n-kitazawa.imi@kyushu-u.ac.jp}
\maketitle
\begin{abstract}
(Co)bordisms of manifolds and maps are fundamental and important objects in algebraic and differential topology of manifolds and related studies were started by Thom etc.. Cobordisms of Morse functions were introduced and have been studied as a branch of the algebraic and differential topological theory of Morse functions and their higher dimensional versions, or the global singularity theory, by Kalm\'{a}r, Ikegami, Sadykov, Saeki, Wrazidlo, Yamamoto etc. since 2000s.

Cobordism relations are in most cases defined as the following for example; two closed manifolds of a fixed dimension or maps on them into a fixed space are said to be {\it cobordant} if the disjoint union is the boundary of a compact manifold or the restriction of a map satisfying suitable conditions on the manifold into the product of the target space and the closed interval. Such relations induce structures of modules consisting of all obtained equivalence classes such that the sums are defined by procedures of taking disjoint unions and in the cases of manifolds, ring structures such that the products are defined by the procedures of taking products, are also introduced. 

In this paper, as a new algebraic topological study, we try to define a product of two cobordism classes of Morse functions and show that a natural method fails, by presenting explicit examples which are regarded as obstructions.

\end{abstract}

\section{Introduction and Preliminaries}
\label{sec:1}
(Co)bordisms of manifolds and maps are fundamental and important objects in algebraic and differential topology of manifolds. In \cite{atiyah}, \cite{stong}, \cite{milnor} etc., fundamental and advanced theory on (co)bordisms are explained and to know precisely, see them. 

In this paper, cobordisms of Morse functions, studied in \cite{ikegami}, \cite{ikegamisaeki}, \cite{saeki3} etc. first and later in \cite{kalmar}, \cite{saekiyamamoto}, \cite{wrazidlo2} etc., are fundamental objects, They are introduced and studied in the context of the algebraic and differential topological theory of Morse functions and their higher dimensional versions, or the global singularity theory. 

Note that before such studies, as good generic maps having no singular points, cobordisms of smooth maps such as immersions and embeddings and versions for maps having singular points have been studied in \cite{wall}, \cite{wells} and later by \cite{rimanyiszucs} etc. and the methods are algebraic and differential topological as the classical studies of cobordisms of manifolds in the beginning of this section. Note also that most of the studies of cobordisms of Morse functions, their higher dimensional versions etc. just before are geometric topological, combinatoric or based on explicit singularity theory such as eliminations of singular points of a type called {\it cusp} (\cite{levine}).

For the algebraic and differential topological theory of Morse functions and their higher dimensional versions, {\it fold} maps are fundamental and important. A {\it fold} map from an $m$-dimensional manifold wth no boundary into an $n$-dimensional manifold with no boundary satisfying the relation $m \geq n \geq 1$ is a smooth map such that the form at each singular point $p$ is 
$$(x_1, \cdots, x_m) \mapsto (x_1,\cdots,x_{n-1},\sum_{k=n}^{m-i_(p)}{x_k}^2-\sum_{k=m-i_(p)+1}^{m}{x_k}^2)$$
for an integer $0 \leq i(p) \leq \frac{m-n+1}{2}$.
As a fundamental property, $i(p)$ is unique. In addition, the set of all singular point of a fixed $i(p)$ is a smooth closed submanifold of dimension $n-1$ and the map obtained by the restriction of the original map to the submanifold is a smooth immersion. Note that a Morse function is a fold map, that at a singular point $p$ of a Morse function $f$, the function is of the form $$(x_1, \cdots, x_m) \mapsto \sum_{k=n}^{m-i_(p)}{x_k}^2-\sum_{k=m-i_(p)+1}^{m}{x_k}^2+f(p)$$ for a uniquely detrmined integer $0 \leq i(p) \leq m$ respecting the orientation of the target space $\mathbb{R}$ (we call it the index of $p$) and the singular points appear discretely. As a well-known fact, Morse functions such that at distinct singular points, the values are distinct, exist densely.   

For fundamental and advanced singularity theory and differential topological properties of Morse functions, fold maps and more general generic maps such as generic smooth maps on manifolds whose dimensions are larger than $1$ into the plane with finite numbers of singular points of the cusp type before (and other singular points are of the type of a singular point of a fold map), see \cite{golubitskyguillemin} and \cite{saeki} for example.

We review a {\it fold cobordism} of Morse functions, introduced and studied in \cite{ikegamisaeki}, \cite{saeki3} etc..
\begin{Def}
\label{def:1}
Two Morse functions $f_1:M_1 \rightarrow \mathbb{R}$ and $f_2:M_2 \rightarrow \mathbb{R}$ on (resp. oriented) closed manifolds of dimension $m>1$  are said to be (resp. {\it oriented}) {\it fold cobordant}.
\begin{enumerate}
\item $M_1$ and $M_2$ is (resp. oriented) cobordant and there exists a smooth (resp. oriented) compact manifold $W$ satisfying $\partial W= M_1 \sqcup M_2$.
\item There exists a smooth map $F:W \rightarrow \mathbb{R} \times [0,1]$ such that the following hold.
\begin{enumerate}
\item The relations $F(\partial W) \subset \mathbb{R} \times \{0,1\}$ and $F({\rm Int} W) \subset \mathbb{R} \times (0,1)$ hold.
\item ${F_i} {\mid}_{M_i}:M_i \rightarrow \mathbb{R} \times \{i-1\} \subset \mathbb{R} \times [0,1]$ holds.
\item For the projection to the first component ${\rm pr}$, the relation $f_i ={\rm pr} \circ {F_i} {\mid}_{M_i}$ holds.
\end{enumerate}
\item $F {\mid}_{{\rm Int} W}:{\rm Int W} \rightarrow \mathbb{R} \times (0,1)$ is a fold map.
\end{enumerate}
\end{Def}

For a closed manifold $M$ of dimension $m \geq 1$ and a Morse function $f$, we denote the number of singular points of index $j$ by $C_j(f)$ and set ${\phi}_j(f):=C_j(f)-C_{m-j}(f)$.
 Let ${\mathcal{N}}_m$ (${\Omega}_m$) be the smooth $m$-dimensional (resp. oriented) cobordism group of closed manifolds and for the (oriented) manifold $M$, we denote the (resp. oriented) cobordism class represented by $M$ by $[M] \in {\mathcal{N}}_m$ (resp. $[M] \in {\Omega}_m$). 
In addition, Definition \ref{def:1} gives a natural commutative group structure on the set of all equivalence classes of functions on closed (oriented) manifolds; the sum is defined by the produce of taking the disjoint union of two functions. We denote the group by ${\mathcal{N}}_{\mathcal{M},m}$ (resp. ${\Omega}_{\mathcal{M},m}$) and call it the {\it $m$-dimensional} (resp. {\it oriented}) {\it cobordism groups of Morse functions}.

For an integer $a$, let $\lfloor a \rfloor$ be the maximal integer not larger than $a$. The following have been shown.

\begin{Thm}[\cite{ikegami}, \cite{kalmar} etc.]
\label{thm:1}
There exists a well-defined isomorphism from ${\mathcal{N}}_{\mathcal{M},m}$ onto ${\mathcal{N}}_m \oplus {\oplus}_{j=\lfloor \frac{m+3}{2} \rfloor}^{m} \mathbb{Z}$ defined by the direct sum of the map corresponding the cobordism class represented by the source manifold and ${\oplus}_{j=\lfloor \frac{m+3}{2} \rfloor}^{m} {\phi}_j(f)$ where $f$ represents a Morse function representing an element of ${\mathcal{N}}_{\mathcal{M},m}$.
\end{Thm}
\begin{Thm}[\cite{ikegami}, \cite{ikegamisaeki} etc.]
\label{thm:2}
Let $m$ be an integer represented by $m=4k+a$ where $k \geq 0$ is an integer and $a=0,1,2,3$.

Let $a=0,2,3$. There exists a well-defined isomorphism from ${\Omega}_{\mathcal{M},m}$ onto ${\mathcal{N}}_m \oplus {\oplus}_{j=\lfloor \frac{m+3}{2} \rfloor}^{m} \mathbb{Z}$ defined by the direct sum of the map corresponding the oriented cobordism class represented by the source manifold and ${\oplus}_{j=\lfloor \frac{m+3}{2} \rfloor}^{m} {\phi}_j(f)$ where $f$ represents a Morse function representing an element of ${\Omega}_{\mathcal{M},m}$.

In the case where $a=1$, for a closed and oriented manifold $M$, let $\sigma(M)={\sum}_{j=0}^{2k} {\rm rank } H_j(M;\mathbb{Q})$. Then, there exists a well-defined isomorphism from ${\Omega}_{\mathcal{M},m}$ onto ${\mathcal{N}}_m \oplus {\oplus}_{j=\lfloor \frac{m+3}{2} \rfloor}^{m} \mathbb{Z}$ defined by the direct sum of the map corresponding the oriented cobordism class represented by the source manifold, ${\oplus}_{j=\lfloor \frac{m+3}{2} \rfloor}^{m} {\phi}_j(f)$ and the map corresponding $\sigma(M)-{\oplus}_{j=0}^{2k} {\phi}_j(f) \in \mathbb{Z}/2\mathbb{Z}$ where $f$ represents a Morse function on a closed manifold $M$ representing an element of ${\Omega}_{\mathcal{M},m}$.
\end{Thm}
 
In this paper, as another algebraic topological study, we try to define a product of elements in these modules and through Theorems \ref{thm:3} and \ref{thm:4}, we explicitly show that a natural method fails.  

Note that such studies were also demonstrated in \cite{kitazawa2} for cobordism-like modules introduced in \cite{kitazawa} and structures were explicitly successfully obtained.

Such a module is defined by regarding two manifolds appearing in a same connected component or adjacent connected components of the set of all regular values of a generic smooth map compatible with simplicial structures of the source and the target manifolds. Such modules have been introduced by the motivation of generalizing a theorem on the topology of {\it Reeb spaces} of maps, defined as the quotient spaces of the source manifolds consisting of all connected components, in considerable cases, being polyhedra of dimensions equal to the dimensions of the target manifolds, inheriting topological information of the source manifolds, and in general being fundamental tools in the theory of global singularity. See \cite{hiratukasaeki}, \cite{hiratukasaeki2} and \cite{kitazawa} for related studies and see \cite{reeb} etc. for fundamental explanations of Reeb spaces.

Throughout the present paper, manifolds and maps between them are smooth and of class $C^{\infty}$ unless otherwise stated. The {\it singular set} of a smooth map is defined as the set of all singular points of the map, the {\it singular value set} of the map is defined as
 the image of the singular set.

The author is a member of the project Grant-in-Aid for Scientific Research (S) (17H06128 Principal Investigator: Osamu Saeki)
"Innovative research of geometric topology and singularities of differentiable mappings"

(https://kaken.nii.ac.jp/en/grant/KAKENHI-PROJECT-17H06128/)
 and supported by the project.
\section{The main theorem}
\label{sec:2}

\begin{figure}
\includegraphics[width=50mm]{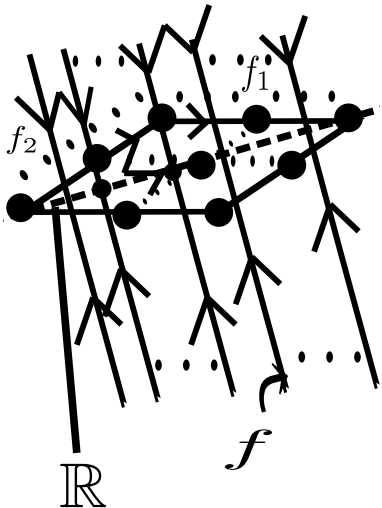}

\caption{Given maps $f_1$ and $f_2$ and their singular values and the map $f$ and its singular values (singular values are represented by dots where for $f$, only some of singular values are represented).}
\label{fig:1}
\end{figure}
\begin{Lem}
\label{lem:1}
Let $f_i:M_i \rightarrow \mathbb{R}$ be a Morse function on a closed manifold $M_i$ of dimension $m_i>0$ for $i=1,2$. Let $F:M \rightarrow {\mathbb{R}}^2$ be a map on the manifold $M=M_1 \times M_2$ defined as {\rm (}the composition of{\rm )} the product $f_1 \times f_2$ of the functions {\rm (}resp. and an affine transformation of the plane{\rm )} and $f$ be the map obtained by composing $F$ and the projection to the straight line including the diagonal canonically oriented and regarded as $\mathbb{R}$. Then $f$ is Morse and a point $p=(p_1,p_2) \in M$ is singular if and only
 if $f_i$ is singular at $p_i$ for $i=1,2$ and the index of the singular point $p$ is regarded as the sum of the index of $p_1$ and that of $p_2$.
\end{Lem}
\begin{proof}
By the definition and FIGURE \ref{fig:1}, we can see that a point on $M$ is a singular point of $f$ if and only if $f_i$ is singular at $p_i$ for $i=1,2$.
For the function $f$, at each singular point, we can regard the Hesse matrix as
$\left(\begin{array}{cc} H_{f_1} & O \\ O & H_{f_2} \end{array} \right)$
where we denote by $H_{f_i}$ the Hesse matrix of the function $f_i$ and by $O$ a zero matrix. The index of the Hesse matrix of $f$ at the singular point, equal to the index of the singular point, is the sum of the indices of the two Hesse matrices, equal to the indices of the singular points of the functions. This completes the proof.
\end{proof}
\begin{Def}
A Morse function $f$ obtained from two Morse functions $f_1$ and $f_2$ by a procedure as in Lemma \ref{lem:1} is said to be 
a {\it diagonal} Morse function obtained from $(f_1,f_2)$.   
\end{Def}
Note that such projections are fundamental and important in the global singularity theory in several explicit situations. For example, Fukuda's formulae in \cite{fukuda} etc., showing relations between Euler numbers of singular sets and the source manifolds of maps explicitly, have been obtained by using such projections and recently in \cite{wrazidlo}, Wrazidlo has obtained a Morse function with just $2$ singular points by composing a canonical projection with a fold map into an Euclidean space whose singular set is diffeomorphic to a standard sphere such that the singular value set is an embedded sphere or {\it standard special generic} map having additional symmetry and restricting the differentiable structure of the source manifold, homeomorphic to a standard sphere, to
 be diffeomorphic to the standard sphere. 

In such a stream, by using Lemma \ref{lem:1}, we have two theorems.
\begin{Thm}
\label{thm:3}
Let $f_i:M_i \rightarrow \mathbb{R}$ be a Morse function on a closed manifold $M_i$ of dimension $m_i>0$ for $i=1,2$ and assume that the relation $m_1 \leq m_2$ holds.
Then for a diagonal Morse function $f$ obtained from $(f_1,f_2)$ and for an integer $0 \leq j < \lfloor m_1 \rfloor$, 
we have \\
 
 ${\phi}_{m_1+m_2-j}(f)={\sum}_{j_1=0}^{j} C_{m_1-j+j_1}(f_1)C_{m_2-j_1}(f_2)-{\sum}_{j_1=0}^{j} C_{j_1}(f_1)C_{j-j_1}(f_2)$.
\end{Thm}
\begin{proof}[The proof of Theorem \ref{thm:3}]
By Lemma \ref{lem:1} and the relations $m_1 \leq m_2$ and $0 \leq j < \lfloor m_1 \rfloor$, we have
\begin{align*}
{\phi}_{m_1+m_2-j}(f) &=c_{m_1+m_2-j}(f)-C_{j}(f) \\
&={\sum}_{j_1=0}^{j} C_{m_1-j+j_1}(f_1)C_{m_2-j_1}(f_2)-{\sum}_{j_1=0}^{j} C_{j_1}(f_1)C_{j-j_1}(f_2)
\end{align*}
and this completes the proof.
\end{proof}
\begin{Thm}
\label{thm:4}
Let $m$ and $m^{\prime}$ be positive integers. For any function $f$ on a closed {\rm (}oriented{\rm )} $m$-dimensional manifold
 satisfying the relation ${\phi}_{m}(f) \neq 0$ and any class of ${\mathcal{N}}_{\mathcal{M},m^{\prime}}$ {\rm (}resp. ${\Omega}_{\mathcal{M},m^{\prime}}${\rm )} such that for a function $f^{\prime}$, there exists an infinite family $\{f_k\}_{k \in \mathbb{N} \sqcup \{0\}}$ of functions in the class such that resulting diagonal Morse functions ${f^{\prime}}_{k_1}$ and ${f^{\prime}}_{k_2}$, obtained from $(f,f_{k_1})$ and $(f,f_{k_2})$, respectively, are in distinct classes of ${\mathcal{N}}_{\mathcal{M},m+m^{\prime}}$ {\rm (}resp. ${\Omega}_{\mathcal{M},m+m^{\prime}}${\rm )} for distinct functions $f_{k_1}$ and $f_{k_2}$. 
\end{Thm}
\begin{proof}
We apply theorem \ref{thm:3} in the case where $j=0$ holds and for Morse functions $f$ and $f_k$, on closed manifolds of dimensions $m$ and $m^{\prime}$, respectively, we have
${\phi}_{m+m^{\prime}}({f^{\prime}}_k)=C_{m}(f)C_{m^{\prime}}(f_k)-C_{0}(f)C_{0}(f_k)$ for the resulting diagonal Morse function ${f^{\prime}}_{k}$ obtained from $(f,f_{k})$. As the family $\{f_k\}$, we construct the maps so that ${\phi}_{m^{\prime}}(f_k)>0$ is invariant for any integer $k \geq 0$ and that for distinct functions $f_{k_1}$ and $f_{k_2}$, the relation $C_0(f_{k_1}) \neq C_0(f_{k_2})$ holds. We can do this by applying a fundamental relation between singular points of a Morse function and handles in a suitable handle decomposition of the source manifold. More precisely, we take a function representing the class and for any positive integer $k>0$, add just $k$ cancelling pairs of $j$-handles, corresponding to singular points whose indices are $j$, and ($j+1$)-handles, corresponding to singular points whose indices are $j+1$, for each $0 \leq j \leq m^{\prime}-1$ and in addition, only in the case where $m^{\prime}$ is represented as $4l+1$ for an integer $l$, one cancelling pair of a $\lfloor \frac{m^{\prime}}{2} \rfloor$-handle and a ($\lfloor \frac{m^{\prime}}{2} \rfloor+1$)-handle, to the source manifold and the function without changing the source manifold. As a result, on the original manifold, we obtain the new function $f_k$ with $2km^{\prime}$ singular points added and $k$ singular points whose indices are $0$ and $m^{\prime}$ added. Theorems \ref{thm:1} and \ref{thm:2} show that the class each $f_k$ belongs to is invariant by such an operation and using the obtained formula ${\phi}_{m+m^{\prime}}({f^{\prime}}_k)=C_{m}(f)C_{m^{\prime}}(f_k)-C_{0}(f)C_{0}(f_k)=C_{m}(f)(C_{m^{\prime}}(f^{\prime})+k)-C_{0}(f)(C_{0}(f^{\prime})+k)$ and the assumption ${\phi}_{m}(f)=C_{m}(f)-C_{0}(f) \neq 0$ in the beginning together with these theorems leads us to the desired result. 
\end{proof}
Theorem \ref{thm:4} means that the class a diagonal Morse function belongs to is not invariant even though the classes original two functions belong to are invariant. We cannot define the class a resulting diagonal Morse function belongs to as the product of the original two classes.

%Note that 
%
%Related to this, for example, we have the following.
%\begin{Thm}
%\label{thm:5}%
%Let $m$ and $m^{\prime}$ be positive integers. For any function $f$ on a closed {\rm (}oriented{\rm )} $m$-dimensional manifold
% satisfying the relation ${\phi}_{m}(f)>0$ and any class of ${\mathcal{N}}_{\mathcal{M},m^{\prime}}$ {\rm (}resp. ${\Omega}_{\mathcal{M},m^{\prime}}${\rm )} such that for a function $f^{\prime}$ representing the class, the relation ${\phi}_{m^{\prime}}(f^{\prime})>0$ holds, there exists an infinite family $\{f_k\}_{k \in \mathbb{N} \sqcup \{0\}}$ of functions in the class such that resulting diagonal Morse functions ${f^{\prime}}_{k_1}$ and ${f^{\prime}}_{k_2}$, obtained from $(f,f_{k_1})$ and $(f,f_{k_2})$, respectively, are in distinct classes of ${\mathcal{N}}_{\mathcal{M},m+m^{\prime}}$ {\rm (}resp. ${\Omega}_{\mathcal{M},m+m^{\prime}}${\rm )} for distinct functions $f_{k_1}$ and $f_{k_2}$. 
%\end{Thm}
%\begin{figure}
%\includegraphics[width=50mm]{aproductofMorse.eps}
%
%\caption{In the case of FIGURE \ref{fig:10}, for vertices $a$ and $b$, we apply the construction.}
%\label{fig:12}
%\end{figure}

\end{document}